\documentclass[11pt, leqno]{article}

\usepackage[utf8]{inputenc}
\usepackage{setspace}
\usepackage{enumitem}
\usepackage{microtype}
\usepackage{graphicx}
\usepackage{amsmath}
\usepackage{amsfonts}
\usepackage{amssymb}
\usepackage{amsthm}
\usepackage{amstext}
\usepackage{braket}
\usepackage{array}
\usepackage{enumitem}
\usepackage{mathrsfs}
\usepackage{ulem}

\usepackage{tikz}
\usepackage{tikz-cd}
\usetikzlibrary{arrows}
\usetikzlibrary{matrix}
\usetikzlibrary{calc}
\usepackage{tikz-qtree}
\usepackage{tikz-qtree-compat}

\usepackage{upgreek }
\usepackage{mathrsfs}
\usepackage{marginnote}
\usepackage{dsfont}

\usepackage{color}
\usepackage{color, ulem}
\usepackage[all]{xy}

\usepackage[english]{babel}
\usepackage{tikz-cd}
\usepackage{color}
\usepackage[backend=biber,style=alphabetic,maxbibnames=99]{biblatex}

\usepackage{color}
\usepackage{enumitem}
\newcommand{\cl}{\mathscr}
\newcommand{\op}{\operatorname}
\newcommand{\clo}{\overline}
\newcommand{\pic}{\op{Pic}}

\theoremstyle{plain}
\newtheorem{theorem}{\textbf{Theorem}}[section]
\newtheorem*{theorem*}{\textbf{Theorem}}
\newtheorem{proposition}[theorem]{\textbf{Proposition}}
\newtheorem{cor}[theorem]{\textbf{Corollary}}	
\newtheorem*{cor*}{\textbf{Corollary}}

\newtheorem*{conj*}{Conjecture}

\theoremstyle{definition}
\newtheorem{definition}[theorem]{\textbf{Definition}}
\newtheorem{remark}[theorem]{\textbf{Remark}}
\newtheorem{example}[theorem]{Example}

\newcommand{\thistheoremname}{}
\newtheorem*{generictheorem*}{\thistheoremname}
\newenvironment{namedtheorem*}[1]
{\renewcommand{\thistheoremname}{#1}%
	\begin{generictheorem*}}
	{\end{generictheorem*}}

\usepackage{hyperref}

\title{Moduli of rank 1 isocrystals}
\date{}
\author{Efstathia Katsigianni\footnote{Freie Universit\"at Berlin, Arnimallee 3, 14195, Berlin, Germany.  Supported by Berlin Mathematical School.  Email: \texttt{efstathia.19@gmail.com}.}} 
\begin{document}
\maketitle
\begin{abstract}
\noindent In this article we express the set of rank 1 isocrystals on a proper curve as a subset of the de Rham moduli space, defined by Simpson and Langer. Using results from the theory of Berkovich spaces, we compare the $\ell$-adic cohomology of this subspace with the $\ell$-adic cohomology of the whole moduli space. This confirms a conjecture of Deligne in the rank 1 case and explains one of his examples in this case from the point of view of isocrystals.
\end{abstract}  

\section{Introduction}
  It is a very natural question  whether  a moduli space which parametrizes isocrystals exists. Using the equivalence of the category of isocrystals on a projective variety with the category of modules with quasi-nilpotent integrable connection on a smooth lift of the variety over the Witt ring, and Grothendieck's formal function theorem, one can see that such a space can be constructed as a subspace of the de Rham moduli space of vector bundles with integrable connection. This was constructed by Simpson in \cite{Simpson1994a} and \cite{Simpson1994} in  characteristic zero and extended by Langer \cite{Langer2014} and others, see for example \cite{Groechenig2016},\cite{Laszlo2001}, \cite{Braverman2007}, in the positive characteristic case.
  
  This article deals with the moduli space of rank 1 isocrystals on a smooth proper curve and provides evidence for a conjecture of Deligne \cite{Deligne2015} in the rank 1 case. 
  Concretely, let $C_0$ be a smooth projective curve over the finite field $\mathbb F_q$ of characteristic $p>0$ and let $C$ be its base change to $\clo{\mathbb F}_q=:k$. 
  Denote by $W$ the Witt ring of $\clo{\mathbb F}_q$, by $K$ its field of fractions, of characteristic 0, and $\clo K$ an algebraic closure of it. Denote by $C_W$ a smooth lift over $\op{Spec}W$ and assume in addition that $C_W$ admits a section $x:\op{Spec}W\to C_W$. 
  
  Deligne conjectured the following in \cite{Deligne2015}:
  let $M_K$ denote the moduli space of vector bundles of rank $r$ on a smooth curve $C$, as before, endowed with an integrable connection, and respectively $M_{\clo{K}}$ the corresponding scheme on an algebraic closure $\clo K$ of $K$. 
  \begin{conj*}[Conjecture 2.18 in \cite{Deligne2015}]
  	The cohomology of $M_{\clo K}$ admits an endomorphism $V^*$, such that for all $n\ge 1$, the number $N_n$ of fixed points of $V^*$ on $E_r$, the set of isomorphism classes of irreducible lisse $\ell'$-adic sheaves of rank $r$ on $C$, is given by
  	\[ N_n=\sum_i(-1)^i\op{Tr}(V^{*n}, H^i(M_{\clo K})),   \]
  	where all cohomology groups denote $\ell$-adic cohomology groups, with $\ell\neq p$. 
  \end{conj*}
  
  Deligne expects that there should exist an open subspace $M^0_{\clo K}$ inside the Berkovich analytification  $M^{\op{an}}_{\clo K}$, which should correspond to the sublocus of isocrystals and such that it has the following two properties:
  \begin{enumerate}[label={(\arabic*)}]
  	\item the restriction morphism $H^*(M_{\clo K}) = H^*(M^{\op{an}}_{\clo K})\to H^*(M^0_{\clo K})$ is an isomorphism and
  	\item a crystalline interpretation of $M^0$ allows us to define $V= \op{Frob}^*: M^0_{\clo K}\to M^0_{\clo K}$, which induces $V^*$ on cohomology.
  \end{enumerate}
  In the same article, Deligne provides an example of this in the rank 1 case, \cite[Example 2.19 and Proposition 2.20]{Deligne2015} using $\pic^0$ as $M_0$. Our goal  is to explain why this example indeed provides a comparison between the cohomology of the subset of isocrystals and that of the moduli space of rank 1 connections. 
  
  The recent work of Hongjie Yu \cite{Yu2018} actually proves Deligne's conjecture and gives explicit formulas for the number of irreducible $\ell$-adic local systems fixed by the Frobenius. In our work, we mainly focus on the first part of this conjecture, specifically on the relation with the theory of isocrystals and subsequently why there should be a Frobenius on the cohomology of the moduli space.
  
  \section*{Leitfaden}
We consider the rank 1 case of Deligne's conjecture and use the moduli space $\op{Pic}^\nabla(C_W)$ of line bundles of degree zero with integrable connection on $C_W$. This is the universal extension of $\op{Pic}^0(C_W)$, see Section \ref{SectionUnivExtension}, and is therefore fine.

Our first goal is to characterize the subspace of rank 1 isocrystals inside this moduli space. Isocrystals can indeed be thought as vector bundles with integrable and nilpotent connection: we say that a point $[(\cl{L},\nabla)]$ of $\op{Pic}^{\nabla}(C_W)$ represents an isocrystal on $C$ if the associated pair 
$(\cl{L}_{\clo{\mathbb F_q}},\nabla_{\clo{\mathbb F_q}})$ represents an isocrystal, or equivalently if
$(\cl{L}_{\clo{\mathbb F_q}},\nabla_{\clo{\mathbb F_q}})$ has nilpotent $p$-curvature. This condition is equivalent to requiring that the characteristic polynomial of the $p$-curvature is zero or that $\upchi([(\cl{L}_{\clo{\mathbb F_q}},\nabla_{\clo{\mathbb F_q}} )])=0$, with $\upchi: \op{Pic}^{\nabla}(C) \to \cl{A}^1(\clo{\mathbb F_q}):=H^0(C,\omega^p_C)$ the Hitchin map.  We denote the fiber over zero of the Hitchin map by $\pic^{\nabla}(C)^{\psi=0}$.
By the Cartier isomorphism, \cite[Theorem 5.1]{Katz1970}, we have that this fiber is isomorphic to $\pic^0(C^{(p)})$, which is actually isomorphic to $\pic^0(C)$.

In the spirit of the above conjecture we consider in Section \ref{SectionModIsoc} the Berkovich analytification of $\pic^{\nabla}(C_W)_K$ and remark that the inverse image of $\pic^{\nabla}(C)^{\psi=0}$ by the reduction map
\[\op{red}: \widehat{(\pic^{\nabla}(C_W))}_K\to \pic^{\nabla}(C)\] is an open subset, denoted by $]\pic^{\nabla}(C)^{\psi=0}[$.
Using results of \cite{Berkovich1990} and \cite{Berkovich1996}, we show in Proposition \ref{PropCoh} that
\[ H^*(\pic^{\nabla}(C_{\clo K})^{\op{an}},\mathbb Q_{\ell})\xrightarrow{\sim}
H^* (]\pic^{\nabla}(C)^{\psi=0}[_{\clo K},\mathbb Q_{\ell})\]
and see therefore that $]\pic^{\nabla}(C)^{\psi=0}[$ is a good candidate for being the open subset $M^0$ of the aforementioned conjecture. 

As a last step, we prove that this subset admits a Frobenius action. To this end, we describe $]\pic^{\nabla}(C)^{\psi=0}[$ as a subfunctor of $\pic^{\nabla}(C_K)^{\op{an}}$. For this, it is enough to characterize the set \[\op{Hom}(S,]\pic^{\nabla}(C)^{\psi=0}[ )\] for $S$ an affinoid. Indeed, using results of \cite{Bosch1993} we obtain in Section \ref{SectionModFrob} an isomorphism between \[\op{Hom}(S,]\pic^{\nabla}(C)^{\psi=0}[ )\] and the set 
\begin{multline}
\{ (\cl{L},\nabla)\text{ line bundles with integrable connection on }S\times_K C^{\op{an}}_K \text{ which are}\\\text{ isocrystals on }
\cl S'\times C_W\to \cl S' \text{ where } \cl S' \text{is a formal model of } S',\\\hfill\text{an admissible formal blow-up of }S  \}\hfill
\end{multline} and prove that this isomorphism is well defined (independent of the choice of an admissible blow-up) and functorial.
 \medskip 
 
 {\it Acknowledgements: }This work is part of my doctoral thesis at the Freie Universit\"at Berlin, which was supervised by  Prof. H\'el\`ene Esnault. The biggest part of this work was completed during my visit in Tokyo, where I was hosted by Prof. Tomoyuki Abe, at Kavli IPMU. I am grateful for his kindness and help. I would like to thank both for their guidance and support during the last three years.

\section*{Notation}
Let $C_0$ be a smooth projective curve over $\mathbb F_q$, where $q=p^n$ for some prime number $p$. We denote by $\clo {\mathbb F_q}$ an algebraic closure of $\mathbb F_q$, and denote by $C$ the curve defined by
\[\begin{tikzcd}
C\arrow{d}\arrow{r}&C_0\arrow{d}\\
\op{Spec}\clo{\mathbb F_q} \arrow{r}&\op{Spec}\mathbb F_q
\end{tikzcd}
\]
Denote by $W$ the Witt ring of $\clo{\mathbb F_q}$ and by $K$ its field of fractions, of characteristic 0, $\clo K$ an algebraic closure of it and by $C_W$ a smooth lift over $\op{Spec}W$.
\begin{equation}
\begin{tikzcd}
C\arrow{r}\arrow{d}{proper,smooth}& C_W\arrow{d}&C_K\arrow{d}\arrow{l}\\
\operatorname{Spec}{\clo{\mathbb F_q}}\arrow[hook]{r}&\operatorname{Spec}W&\operatorname{Spec}K\arrow{l}.
\end{tikzcd}
\end{equation}
\section{Background}
In this section we include some background on the theory of Berkovich and rigid analytifications and refer to results that we use later, in order to make this article as self contained as possible.
\subsection{Berkovich and rigid analytifications}
After the introduction of overconvergent isocrystals by Berthelot in \cite{BerPre}, rigid analytic geometry became a necessary tool for the study of isocrystals. The theory of Berkovich spaces is another approach to non-archimedean geometry and provides us with spaces that are very close to rigid analytic spaces, but have a true topology, in contrast to the rigid analytic ones, which have a Grothendieck topology. We focus mainly on the functors of analytification from algebraic schemes to the respective categories of rigid analytic and Berkovich spaces. For the basic definitions in the theory of rigid analytic geometry we refer to \cite[Chapter 1-5]{Bosch1993}. Nice introductions to the theory of Berkovich spaces can be found in \cite{Baker2008},\cite{Temkin2015},\cite{Ducros2015} and of course in the original papers of Berkovich \cite{Berkovich1990} and \cite{Berkovich1993}.

In this section we denote by $K$ a complete field equipped with a non-trivial non-archimedean absolute value, $R$ denotes the associated valuation ring, and $k$ the residue field of $R$.
\subsection{Rigid analytification}
Building blocks of rigid analytic spaces are affinoid spaces, \cite[Section 3.2]{Bosch2014}. They are defined as maximal spectra of affinoid $K$-algebras, which in turn are defined as quotients of Tate algebras associated to $K$, for details see \cite[]{Bosch2014}. 
On an affinoid $K$-space $X$, a Grothendieck topology $\mathscr{T}$ is defined  in the following way: the objects of $\op{Cat}\mathscr T$ are affinoids subdomains, morphisms are inclusions and coverings are the finite coverings by affinoid subdomains. 
From this local data, we can construct a Grothendieck topology $\mathscr T$ on a $K$-space $X$. Objects of $\op{Cat}\mathscr{T}$ are called \textit{admissible open subsets} of $X$ and coverings in $\mathscr T$ are called \textit{admissible coverings}.

The rigid analytification of a $K$-scheme is defined via a universal property:
\begin{definition}[Rigid analytification]\label{DefRigAn}
	The \textit{rigid analytification} of a $K$-scheme $(X,\mathcal O_X)$ locally of finite type is a rigid analytic space $(X^{\op{an}},\mathcal O_{X^{\op{an}}})$ together with a morphism of locally G-ringed $K$-spaces $(i,i^*):(X^{\op{an}},\mathcal O_{X^{\op{an}}})\to (X,\mathcal O_X)$ which has the following universal property: any morphism of locally G-ringed $K$-spaces $(Y,\mathcal O_Y)\to (X,\mathcal O_X)$, for $(Y,\mathcal O_Y)$ a rigid $K$-space, factors through $(i,i^*)$.
\end{definition}
If we have now a morphism of $K$-schemes $f:X\to S$, we obtain its analytification $f^{\op{an}}$ of $f$ by the universal property of analytification:
\begin{equation*}
\begin{tikzcd}
X^{\op{an}}\arrow{r}{f^{\op{an}}}\arrow{d}& Y^{\op{an}}\arrow{d}\\
X\arrow{r}{f}&Y\,\,\,.
\end{tikzcd}
\end{equation*}
Moreover, one can show as in \cite[Proposition 5.4.4]{Bosch2014}, that every $K$-scheme admits an analytification, which is constructed by gluing the analytifications of its affine parts, and that the map of sets $i:X^{\op{an}}\to X$ identifies the closed points of $X^{\op{an}}$ and $X$.

Given a coherent sheaf  $\cl{F}$ on an algebraic variety $X$, one can define a coherent sheaf of $\mathcal O_{X^{\op{an}}}$-modules $\cl{F}^{\op{an}}$, by $\cl{F}^{\op{an}}:=i^{-1}\cl{F}\otimes_{i^{-1}\mathcal O_X}\mathcal O_{X^{\op{an}}}$. Moreover, there is an analogue of Serre's GAGA theorem in the rigid analytic case:
\begin{theorem}[Rigid GAGA {\cite[Theorem 6.3.13]{Bosch2014}}]
	If $X$ is proper over $K$, then the functor  $\cl{F}\mapsto \cl{F}^{\op{an}}$ induces an equivalence of categories between the category of coherent $\mathcal O_X$-modules and the category of coherent $\mathcal O_{X^{\op{an}}}$-modules. Moreover, for every $n\ge 0$, there exist a canonical isomorphism
	\[\op{H}^n(X,\cl{F})\simeq \op{H}^n(X^{\op{an}},\cl{F}^{\op{an}}). \]
\end{theorem}

We will be particularly interested in the rigid analytification of $\pic^0$:
\begin{remark}[{\cite[Lemma 4.3.2]{Conrad2006}}]\label{AnalytPic}
	Given a map between rigid spaces $X\to S$, where $X$ has geometrically reduced and connected fibers over $S$, together with a section $x\in X(S)$, we can consider the functor $\textbf{Pic}_{X/S,x}$, which associates to any rigid analytic spaces $S'$ over $S$, the set of $x_{S'}$-rigidified line bundles on $X\times_S S'$.
	
	When the rigid spaces are rigid analytifications of algebraic varieties, this functor is actually representable: for a map $X\to S$ of algebraic $K$ schemes with geometrically reduced and connected fibers, it is proven in \cite[Lemma 4.3.2]{Conrad2006} that the analytification of the Picard scheme $\op{Pic}_{X/S,x}$ represents the functor
	$\textbf{Pic}_{X^{\op{an}}/S^{\op{an}}, x^{\op{an}}}$. 
\end{remark}	

\subsection{Berkovich analytification}
Rigid analytic geometry makes it possible to define the notion of an analytic function on a non-archimedean field $K$, but doesn't provide a good topological space to work with; one works with a Grothendieck topology instead. Berkovich's $K$-analytic spaces though have a true topology, which makes working with them easier.

Building blocks of a $K$-analytic space are the so called $K$-affinoid spaces: spaces of the form $\cl{M}(A)=$ the space of multiplicative semi-norms on the Banach algebra $A$, see \cite[Definition 4.1.2.1]{Ducros2015}. A $K$-analytic space is then a topological space which admits a $K$-affinoid atlas on $X$, together with a family of subsets, called a net, see \cite[Definition 4.1.1.1]{Ducros2015}.

As in the case of rigid analytic spaces, we are mainly interested in the functor which assigns a $K$-analytic space to an algebraic $K$-variety. This functor is called the \textit{Berkovich analytification functor}.
This functor is defined as follows on affine schemes: let $X=\op{Spec}K[T_1,\cdots,T_n]$. We set $X^{\op{an}}:=\cl{M}(K[T_1,\cdots,T_n])$, which is exactly the affine space of the previous example. For a scheme $\op{Spec}K[T_1,\cdots,T_n]/I$, we define the analytification as the closed subset of $X^{\op{an}}$ defined by the vanishing of $I\mathcal O_{X^{\op{an}}}$.

The analytification functor can also be defined via a universal property: the analytification of an algebraic variety $X$ represents the functor which assigns to a $K$-analytic space $Y$, the set of morphisms of locally ringed spaces $Y\to X$. This also means, exactly like the rigid analytic case, that there is a morphism $i:X^{\op{an}}\to X$. For a coherent $\mathcal O_X$-module $\cl F$, we can define a coherent $\mathcal O_{X^{\op{an}}}$-module $\cl{F}^{\op{an}}=i^*(\cl F)$.

\begin{remark}\label{RemBerkRig}
	There is a close relation between $K$-analytic spaces and rigid analytic $K$-spaces, which is explained in \cite[Section 3.3]{Berkovich1990}.
	To state it, we first need to recall the following: for $\cl{A}$ a commutative Banach algebra over $K$, a point $x\in \cl{M}(\cl{A})$ gives rise to a character $\cl{A}\to \cl{H}(x)$, where $\cl{H}(x)$ is a field. Indeed, the multiplicative semi-norm corresponding to $x$ extends to the fraction field of the quotient ring of $\cl{A}$ by its kernel and $\cl{H}(x)$ is then defined as the completion of this field, see \cite[p.4]{Berkovich1990}.
	
	By \cite[Proposition 2.2.5]{Berkovich1990}, the class of strictly affinoid domains in $X$ defines a $G$-topology on a separated strictly $K$-analytic space $X$ and setting $X_0=\{x\in X|[\cl{H}(x):K]<\infty \}$, this has the induced $G$-topology and has the structure of a rigid analytic $K$-space. Moreover, \cite[Proposition 2.3.1]{Berkovich1990} implies that given a sheaf $\cl{F}$ on $X$, we get an induced sheaf $\cl{F}_0$ on $X_0$ and the correspondence $\cl{F}\mapsto \cl{F}_0$ is an equivalence between the categories of coherent sheaves on $X$ and $X_0$.

	As explained also in \cite[p.58]{Baker2008}, by the correspondence $X\mapsto X_0$ we also obtain an equivalence of categories between the category of quasi-separated rigid spaces admitting a locally finite admissible covering by affinoid opens and the category of paracompact Hausdorff strictly $K$-analytic spaces. Moreover, the Berkovich analytification functor from algebraic $K$-schemes to strictly analytic $K$-spaces and the rigid analytification functor from algebraic $K$-schemes to rigid analytic spaces over $K$ is compatible with this equivalence. This functor in fact preserves the category of locally constant sheaves and their cohomology groups, by \cite[Proposition 3.3.4]{Berkovich1990}. Therefore, cohomological results from the theory of rigid analytic spaces are applicable to Berkovich's $K$-analytic spaces.
\end{remark}

\subsection{Generic fibers of formal schemes}
It is possible to relate both the Berkovich and the rigid analytification of an algebraic variety with the theory of formal schemes; more specifically their generic fibers.  
\subsection*{Berkovich spaces and generic fibers of formal schemes}
Let $K$ be a non-archimedean field with a non-trivial discrete valuation, $R$ the ring of integers of $K$ and $k$ its residue field. For a special $R$-formal scheme $\cl{X}$, the ringed space $(\cl{X},\mathcal O_{\cl{X}}/\cl{I})$, where $\cl{I}$ is an ideal of definition of $\cl{X}$ that contains the maximal ideal $m_R$ of $R$, is a scheme locally of finite type over $k$ and is called the \textit{closed fiber} of $\cl{X}$, denoted by $\cl{X}_s$. It depends on the choice of $\cl{I}$ but the underlying reduced scheme and the associated topos do not, see \cite[p.370]{Berkovich1996}. 

\begin{remark}[{\cite[p.370]{Berkovich1996}}]
	If $\cl{Y}\subset \cl{X}_s$, the formal completion of $\cl{X}$ along $\cl{Y}$ is a special formal scheme over $R$. We denote this completion by $\cl{X}_{|\cl{Y}}$.
\end{remark}
Berkovich defines in \cite{Berkovich1996} a functor from the category of special formal $R$-schemes to the category of $K$-analytic spaces, which associates to $\cl{X}$ its generic fiber $\cl{X}_\eta$.
For a formal scheme of the form $\cl{X}=\op{Spf}A$, with $A$ a special $R$-algebra, the generic fiber $\cl{X}_\eta$ is defined as the set of continuous multiplicative semi-norms on $A$ that extend the valuation on $R$ and have value at most 1. For an arbitrary formal scheme $\cl{X}$, one can take an affine covering ${\cl{X}_i}$ and glue $\cl{X}_{i,\eta}$. For all details we refer to \cite[p.370-371]{Berkovich1996}.

\begin{example}
	In the case when $\cl{X}=\op{Spf}A$, with $A=K\{T_1,\cdots,T_n  \}[[S_1,\cdots,S_m]]$, where $K\{T_1,\cdots,T_n  \}$ is the algebra of restricted power series with coefficients in $K$. Then one defines $\cl{X}_\eta=E^m(0,1)\times D^n(0,1)$, where 
	$E^m(0,1)$ and  $D^n(0,1)$ are the closed and open polydiscs of radius 1 and center at zero, in $\mathbb A^m$ and $\mathbb A^n$ respectively.
\end{example}	
Moreover, one can construct a map, called the \textit{reduction map} $\pi:\cl{X}_{\eta}\to \cl{X}_s$.
\begin{proposition}[Definition of reduction map]\label{DefReduc}
	There is an anticontinuous map $\pi:\cl{X}_{\eta}\to \cl{X}_s$, called the reduction map.
\end{proposition}
The proof of this fact can be found in \cite[p. 541]{Berkovich1996a}.

One can describe the inverse image of a subscheme $\cl{Y}$ of $\cl{X}_s$ in more detail:
\begin{proposition}[{\cite[Proposition 1.3]{Berkovich1996}}]\label{PropInverseImageBerk}
	With notation as in the above proposition, there is a canonical isomorphism $\pi^{-1}(\cl{Y})\simeq (\cl{X}_{|{\cl{Y}}})_\eta$.
\end{proposition}
\begin{remark}[{\cite[p.553]{Berkovich1996a}}]\label{RemBerkGen}
	Given now a scheme of finite type over $R$, one can relate to it a $K$-analytic space in two different ways: by associating to it the analytification of its generic fiber $(X_\eta)^{\op{an}}$ and the generic fiber of its completion $(\widehat{X})_\eta$, the construction of which we summarized in this section. By construction, there is a morphism $(\widehat{X})_\eta \to (X_\eta)^{\op{an}}$ and for $X$ separated and finitely presented, this identifies $(\widehat{X})_\eta$ with a closed analytic subvariety of $(X_\eta)^{\op{an}}$; in the affine case $X=\op{Spec}A$ where $A$ is generated by $f_1,\cdots,f_n$, $(\widehat{X})_\eta=\{ x\in (X_\eta)^{\op{an}}| |f_i(x)|\le 1 \}$. For arbitrary varieties, one takes a finite covering by open affine subschemes and concludes in a similar way. 
	If $X$ is proper, this morphism is in fact an isomorphism. This follows from the more general fact, that a proper morphism $\phi:Z\to X$ induces an isomorphism $\widehat{Z}_\eta \overset{\simeq}{\longrightarrow} Z^{\op{an}}_\eta \times_{X^{\op{an}}_\eta} \widehat{X}_\eta$, applied to $X$ being a point.
\end{remark}

\subsection*{Formal models for rigid spaces}
Just as for Berkovich spaces, we can relate rigid spaces to generic fibers of formal schemes. 
%
In order to do this, we need the notion of an admissible formal blowing up:
\begin{definition}
	Let $\cl{X}/\cl{S}$ an admissible formal $\cl{S}$-scheme with ideal of definition $\cl{I}\mathcal O_{\cl{X}}$ and $\cl{A}\subset \mathcal O_{\cl{X}}$ an open ideal. Then we define
	\[ \cl{X}':=  \varinjlim_m \op{Proj}\oplus_{n=0}^{\infty}(\cl{A}^n\otimes_{\mathcal O_{\cl{X}}} \mathcal{O}_{\cl{X}}/\cl{I}^{m+1})   \]
	and call the morphism of $\cl{S}$-formal schemes $\phi:\cl{X}'\to \cl{X}$, an \textit{admissible formal blowing up} with respect to $\cl{A}$.
\end{definition}	
Among many interesting properties, see for example \cite[Section 2]{Bosch1993}, an admissible formal blowing up satisfies a universal property,\cite[Proposition 2.1(c)]{Bosch1993}: if $\psi:\cl{Z}\to \cl{X}$ is a morphism of formal $\cl{S}$-schemes such that $\cl{A}\mathcal O_{\cl{X}}$ is invertible on $\cl{Z}$, there exists a unique $\cl{S}$-morphism $\psi':\cl{Z}\to \cl{X}'$ such that $\psi=\phi\circ \psi'$. Note that \cite[Lemma 2.2]{Bosch1993} gives an explicit description of an admissible blowing up in the case of an affine formal scheme.

The relation between formal schemes and rigid analytic spaces can be seen in the next theorem of Raynaud:
\begin{theorem}[{\cite[Theorem 4.1]{Bosch1993}}, \cite{Raynaud1974}]
	There is an equivalence of categories between:
	\begin{enumerate}[label={(\arabic*)}]
		\item the category of quasi-compact admissible formal schemes, localized by admissible formal blowing ups, and
		\item the category of rigid $K$-spaces, which are quasi-compact and quasi-separated.
	\end{enumerate}
\end{theorem}
In the affine case the above functor is defined as follows: to an admissible $R$-algebra $A=R\langle\xi\rangle/\mathfrak{a}$ we associate the affinoid $K$-algebra $A_{\op{rig}}:=A\otimes_R K=R\langle\xi\rangle/\mathfrak{a}K$. Globalizing this construction, as explained in \cite[Section 4]{Bosch1993}, we obtain a functor from admissible formal $R$-schemes to the category of rigid $K$-spaces. 
We denote this functor by $\cl{X}\mapsto \cl{X}_{\op{rig}}$. We call $\cl{X}_{\op{rig}}$ the \textit{rigid analytic fiber} of $\cl{X}$.

In the inverse direction and in the affinoid case, this functor is given by associating to $X_K=\op{Sp}A_K=\op{Sp}K\langle\xi\rangle/\mathfrak{a}K$, the affine formal scheme $\cl{X}=\op{Sp}R\langle\xi\rangle/\mathfrak{a}\cap R\langle \xi\rangle$.
\begin{remark}\label{FormModProof}
	In Section \ref{SectionModFrob} we use in particular some properties of this equivalence, found in \cite[Proof of Theorem 4.1, (a), (b) and (c)]{Bosch1993}:
	\begin{enumerate}[label={(\arabic*)}]\label{FormalRigid}
		\item  the above functor sends admissible blowing ups to isomorphisms: if $\phi:\cl{X}'\to \cl{X}$ is an admissible formal blowing up, then $\phi_{\op{rig}}:\cl{X}'_{\op{rig}} \to \cl{X}_{\op{rig}}$ is an isomorphism.
		\item Two morphisms $\phi,\psi:\cl{Z}\to \cl{X}$ of formal $R$-schemes coincide if $\phi_{\op{rig}}$ and $\psi_{\op{rig}}$ coincide. 
		\item If $\phi_K:\cl{Z}_{\op{rig}}\to \cl{X}_{\op{rig}}$ is a morphism between the rigid analytic fibers of two formal $R$-schemes, then there is an admissible formal blowing up $r:\cl{Z}'\to \cl{Z}$ and a morphism $\phi_{\op{rig}}:\cl{Z}'\to \cl{X}$ such that $\phi_{\op{rig}}=\phi_K\circ r$.
	\end{enumerate}
\end{remark}

\subsection{Cohomology of Berkovich spaces}\label{SubsectionBerCoh}
We state here some results from \cite[Section 5]{Berkovich1996a}. Let $X$ be a scheme of finite type over $\op{Spec}R$, with notations as in the previous sections. We have a diagram, where $X_\eta$ and $X_s$ are the generic and closed fibers respectively:
\begin{equation*}
\begin{tikzcd}
X_\eta\arrow[hook]{r}{j}&X&X_s\arrow{l}{i}\\
X_{\clo{\eta}}\arrow[hook]{r}{\clo j}\arrow{u} &\clo{X}\arrow{u}&X_{\clo{s}}\arrow{u}\arrow{l}{\clo i}.
\end{tikzcd}
\end{equation*}
The \textit{vanishing cycles functor} $\Psi_\eta:\widetilde{X}_{\eta,\op{ \acute{e}t}}\to \widetilde{X}_{\clo{s},\op{ \acute{e}t}}$ is then defined as $\Psi_\eta(\cl F)= \clo{i}^*(\clo{j}_*\clo{\cl F})$.

Let now $\cl{Y}\subset X_s$ be a subscheme of the closed fiber of $X \to \op{Spec}R$. Then the formal completion of $X$ along $\cl{Y}$, $\widehat{X}_{|\cl{Y}}$ is a special formal scheme, in the sense of \cite{Berkovich1996}, whose closed fiber is identified with $\cl{Y}$. By \cite[Proposition 1.3]{Berkovich1996}, there is a canonical isomorphism
\[ (\widehat{X}_{|\cl{Y}})_\eta \simeq \op{\pi}^{-1}(\cl{Y}). \]

As recalled in Remark \ref{RemBerkGen}, we have a morphism $\widehat{X}_\eta \to X^{\op{an}}_\eta$. Moreover, there are canonical morphisms of sites $\widehat{X}_{\eta,\text{\'et}} \to X^{\op{an}}_{\eta,\text{\'et}} \to X_{\eta,\text{\'et}}$. For details about the \'etale topos on an analytic space, we refer to \cite[Section 3]{Berkovich1996a}. For a sheaf $\cl{F}$ on ${X}_{\eta,\text{\'et}}$, we denote by $\cl{F}^{\op{an}}$ and $\widehat{\cl{F}}$ the respective pullbacks to $X^{\op{an}}_{\eta,\text{\'et}}$ and $\widehat{X}_{\eta,\text{\'et}}$, while there is also a canonical morphism of sheaves $\Psi_\eta(\cl{F})\to \Psi_\eta(\widehat{\cl F})$.
For this morphism more is in fact true:
\begin{proposition}[{\cite[Corollary 5.3]{Berkovich1996a}}]
	If $\cl{F}$ is an abelian torsion sheaf on $X_\eta$, then for any $q\ge 0$, there is a canonical isomorphism 
	\[ R^q\Psi_\eta(\cl{F})\overset{\simeq}{\longrightarrow} R^q\Psi_\eta(\widehat{\cl F}). \]
\end{proposition}
\begin{cor}[{\cite[Corollary 5.4]{Berkovich1996a}}]
	If $\cl{X}$ is a smooth formal $R$-scheme and $n$ is an integer coprime to $p$, then 
	\[\Psi_\eta(\mathbb{Z}/ n \mathbb{Z})_{\cl{X}_\eta}
	=(\mathbb{Z}/  n\mathbb{Z})_{\cl{X}_{\overline{s}}}
	\,\,\,  \text{ and } \,\,\, R^q\Psi_\eta(\mathbb{Z}/ n \mathbb{Z})_{\cl{X}_\eta} =0,  q\ge 1 .\]
\end{cor}
We are mainly interested in comparing the $\ell$-adic cohomology of a subset $\cl{Y}$ of $ X_s$ and the cohomology of its inverse image by the reduction map $\op{\pi}^{-1}(\cl{Y})$. 
The following theorem of Berkovich provides us with a comparison between these cohomology groups:
\begin{theorem}[{\cite[Theorem 3.1]{Berkovich1996}}]
	Let $\cl F$ be an abelian torsion sheaf on $X_\eta$ with torsion orders prime to $p$. Then for $\cl{Y}\subset X_s$ an open in the closed fiber $X_s$ we have
	\[ (\op{R}^q\Psi_\eta\cl F)_{|\cl{Y}} \overset{\simeq}{\longrightarrow} \op{R}^q \Psi_\eta(\widehat{\cl F}_{|\cl{Y}}), \,\,\,\,\,\forall q\ge 0, \]
	where $\widehat{\cl F}$ is the pullback of $\cl F$ to $\widehat{X}_\eta$.
\end{theorem}
As a corollary of this one gets:
\begin{cor}[{\cite[Corollary 3.5]{Berkovich1996}}]
	Let $X,\cl{Y}$ be as before and $\cl F$ be a constructible sheaf on $X_\eta$ with torsion orders prime to $p$. Then there are canonical isomorphisms
	\[ \op{R\Gamma}(\cl{Y},\op{R\Psi_\eta}\cl F)\overset{\simeq}{\longrightarrow} \op{R\Gamma}(\op{\pi}^{-1}(\cl{Y}) , \cl F^{\op{an}}) . \]
\end{cor}
In particular:
\begin{cor}[{\cite[Corollary 3.7]{Berkovich1996}}]\label{CorBerkCoh}
	With notations as before, there are isomorphisms	\[ \op{H}^q(\cl{Y}, \mathbb Z/n\mathbb{Z})  \overset{\simeq}{\longrightarrow} \op{H}^q(\op{\pi}^{-1}(\cl{Y}) , \mathbb Z/n\mathbb{Z}) ,\,\,\,\,\,\forall q\ge 0, n\not|p .\]
\end{cor}
\section{Moduli of line bundles with  integrable connection}\label{SectionUnivExtension}
Let $X$ be an $S$ scheme, with $S$ a noetherian scheme over $W$, together with a section $x: S\to X$. In general, we can define a functor
\begin{equation}
\op{Pic}^{\#}_{X/S}: (S-\op{Sch}) \to (\op{Groups})
\end{equation} which associates to
an $S$-scheme $T$, the group of isomorphism classes of line bundles $\cl{L}$ on $T\times_S X$ endowed with an integrable connection $\nabla: \cl{L}\to \cl{L}\otimes \op{pr}_X^*\Omega^1_{X/S}$, such that $\cl{L}_{|x\times S}\simeq \cl{O}_S$, where $\op{pr}_X:X\times_S T \to X$ is the projection.
We denote by $\op{Pic}^{\#}(X/S)$ the group of isomorphism classes of  such pairs on $X/S$. As remarked in \cite[(2.5.3)]{Messing1973}, there is an identification 
\[ \op{Pic}^{\#}(X/S)= \op{H}^1(X,\cl{O}_{X} \xrightarrow{\op{dlog}} \Omega^1_{X/S} \to \Omega^2_{X/S} \to \cdots) \]
and then if $X\times_S T=:X_T$ for an $S$- scheme $T$ admits a section over $T$, we also have \cite[(2.6.4)]{Messing1973} 
\begin{equation}\label{PicCoker}
\op{Pic}^{\#}_{X/S}(T) =\op{Coker}(\op{Pic}(T)\xrightarrow{f^*}\op{Pic}^{\#}(X\times_S T/T)).
\end{equation}

In characteristic zero, this functor can be immediately seen to be representable by a quasi-projective scheme:  in \cite[Sections 2.3.3, 2.3.5]{Bost2013} it is actually remarked that for bundles of any rank $r$, this functor is the same as the representation functor $R_{\op{DR}}(X,x,r)$, which is defined in \cite[p. 55 and Theorem 6.13]{Simpson1994}. There, it is also proven that this functor is  representable.

We consider here the subfunctor of degree 0 line bundles as above, with integrable connection on the curve $X/S$, and denote it by $\op{Pic}^{\nabla}_{X/S}$.
In this case we have that this functor is representable also in positive characteristic, since it is the universal vector extension of $\op{Pic}^0_{X/S}$, as was first proven in \cite[Proposition 2.8.1]{Messing1973} and \cite[Theorem 2.6 and 3.2.3]{Mazur1974}. In more detail:

\begin{definition}
	The universal vector extension of an abelian scheme $A$ over $S$  is an abelian scheme $E$ over $S$ sitting in an exact sequence of fppf sheaves
	\begin{equation}\label{univ1}
	0\to V \to E \to A \to 0,
	\end{equation}
	where $V$ is a vector group over $S$, with the following universal property: for any other abelian scheme $E'$ and vector group $V'$ sitting in an exact sequence of fppf sheaves of abelian groups
	\begin{equation}\label{univ2}
	0\to V' \to E' \to A \to 0
	\end{equation}
	there is an $\cl{O}_S$-linear morphism of abelian groups $ \psi: V\to V'$ such that (\ref{univ2}) is isomorphic to the pushout of  (\ref{univ1}) along $\psi$.
\end{definition}
\begin{remark}\cite[(2.6)]{Mazur1974}
	Assume we have an extension \[0\to \mathbb{G}_m\to E \to A\to 0\] of an abelian variety $A$ over a scheme $S$  by $\mathbb{G}_m$. By \cite[(2.2.1)]{Mazur1974} there is an exact sequence:
	\begin{equation}\label{unextdual}
	0\to \omega_A \to \mathbb{E}_A \to \underline{\op{Ext}^1}(A,\mathbb{G}_m)\to 0,
	\end{equation}
	where $\omega_A$ denotes as in \cite{Mazur1974} and \cite{Messing1973} the module of invariant differentials of $A$.
	In \cite[(2.6)]{Mazur1974} it is proven however that, since $\op{Ext}^1(A,\mathbb{G}_m)$ is isomorphic to the dual abelian variety $A^*$, $\mathbb{E}_A$ is representable by a smooth $S$ -group scheme and (\ref{unextdual}) gives the universal extension of $A^*$.
\end{remark}

In the case of a relative curve $C_W/\op{Spec}W$, which admits a section, the same is true: as remarked in \cite[Section 3]{Katz2014} we have an exact sequence
\begin{equation}\label{univExt}
0\to \op{H}^0(C,\Omega^1_{C_W/W}) \to \op{Pic}^{\nabla}(C_W) \to \op{Pic}^0(C_W)\to 0.
\end{equation}
Indeed, every line bundle $\cl{L}$ on $C_W$ over $\op{Spec}W$ which is fiber by fiber of degree 0 admits an $S$- linear connection, \cite[p. 46]{Mazur1974}. Moreover, if we have two connections $(\cl{L},\nabla_1)$ and $(\cl{L},\nabla_2)$, the difference of $\nabla_1$ and $\nabla_2$ is an element $\omega\in \op{H}^0(C,\Omega^1_{C_W/W})$.

In this case however, the dual abelian scheme $A^*$ of the above Remark is isomorphic to the Jacobian $\op{Pic}^0(C_W)$ and we get, by \cite[Theorem 3.2.3]{Mazur1974} that (\ref{univExt}) is the universal extension of $\op{Pic}^0(C_W)$. All details about this case are presented in \cite[Appendix]{Bost2009} and recalled in \cite[Section 2]{Laumon1996}.

We can see the same in a concrete example in the classical case:
\begin{example}\label{ExampleClassMess}[{\cite[(3.0)]{Messing1973}}]
	If $X$ is a non-singular, connected curve over $\mathbb C$, we denote by $J$ its Jacobian
	and choose a canonical (Abel-Jacobi) map $X\to J$. This map induces a map between the corresponding exact sequences on $X$ and $J$, obtained by \cite[(2.6.4)]{Messing1973}:
	\begin{equation}
	\begin{tikzcd}
	0\arrow{r}& \op{H}^0(\Omega^1_X)\arrow{r}& \op{Pic}^\#(X)\arrow{r}&\op{Pic}(X)\arrow{r}& \op{H}^1(\Omega^1_X)\\
	0\arrow{r}& \op{H}^0(\Omega^1_J)\arrow{r}\arrow{u}{\simeq}& \op{Pic}^\#(J)\arrow{r}\arrow{u}&\op{Pic}(J)\arrow{r}\arrow{u}&\op{H}^2(\tau(\Omega^\bullet_J))\arrow{u}
	\end{tikzcd}
	\end{equation}
	where $\tau(\Omega^\bullet_J)$ is the complex $\Omega^1_{X/S} \to \Omega^2_{X/S} \to \cdots$ with $\Omega^1_{X/S} $ is in degree 1. Note that we use here that the global 1-forms on $X$ and $J$ are all closed, since $X$ is a curve and $J$ is smooth and projective. 
	An element in the image  of $\pic^{\#}\to \pic$ has to be a torsion element in $H^2(X,\mathbb Z)$ and $\op{H}^2(J,\mathbb Z)$ respectively. Hence, since $\op{H}^2(X,\mathbb Z)$ and $\op{H}^2(J,\mathbb Z)$ are torsion free,
	its Chern class is zero and therefore element has to be inside $\pic^0$. Moreover, we have an isomorphism $\pic^0(J)\overset{\simeq}{\longrightarrow} \pic^0(X)$. Hence,  $\pic^\#(X) \overset{\simeq}{\longrightarrow} \pic^\#(J)$, which is the universal extension of $\pic^0(X)$, by the abelian variety case, see \cite[Proposition 2.9.2]{Messing1973}.
\end{example}

\section{The subset of isocrystals.}\label{SectionModIsoc}

As the previous section shows, we get a fine moduli space of line bundles with integrable connection, as described above, represented by a $W$-group scheme $\op{Pic}^{\nabla}(C_W/\op{Spec}W)$, which we denote by $\op{Pic}^{\nabla}(C_W)$ for simplicity.
We then have $\op{Pic}^{\nabla}(C_W)\otimes K \simeq \pic^{\nabla}(C_K)$ and $\op{Pic}^{\nabla}(C_W)\otimes {\clo{\mathbb F_q}}\simeq \pic^{\nabla}(C)$ for its generic and closed fibers, as well as $\op{Pic}^{\nabla}(C_W)\otimes \clo K \simeq \pic^{\nabla}(C_{\clo K})$.

\subsection{Hitchin map}
In the following let $X$ be a smooth scheme over $S$, which is a scheme over a field of characteristic $p>0$. We denote the absolute Frobenius of $S$ by $F_S:S\to S$ (i.e. the $p$-th power mapping on $\mathcal O_S$) and by $F_{X/S}:X\to X^{(p)}$ the relative Frobenius, which is defined by the cartesian diagram:
\begin{equation*}
\begin{tikzcd}
X\arrow{r}{F_{X/S}}\arrow{dr}&X^{(p)}\arrow{r}\arrow{d}&X\arrow{d}\\
&S\arrow{r}{F_S}&S.
\end{tikzcd}
\end{equation*}

We denote by $\op{MIC}(X/S)$ the abelian category of $\mathcal O_X$-modules with integrable connection on $X/S$.
Given an element $(\cl{E},\nabla)$ of $\op{MIC}(X/S)$, we denote by $\cl E^{\nabla}$ the kernel of $\nabla$.

For an element $(\cl{E},\nabla)$ of $\op{MIC}(X/S)$, Katz defines in \cite[Section 5]{Katz1970} its $p$-curvature by 
\begin{equation}\label{pCurvature}
\op{Der}(X/S) \to \op{End}_S(\cl{E}) ;\,\,\, D\mapsto (\nabla(D))^p - \nabla(D^p).
\end{equation}
The morphism $\psi(\nabla)$ is $p$-linear \cite[Proposition 5.2]{Katz1970}: it is additive and $\psi(\nabla)(fD)=f^p\psi(\nabla)(D)$, for $f$ and $D$ local sections of $\cl{O}_X$ and $\op{Der}(X/S)$ over an open subset of $S$. Using $p$-linearity, we can consider the $p$-curvature as a global section in $\op{H}^0(X,\op{End}(\cl{E})\otimes F^*_{X/S}\Omega^1_{X^{(p)/S}})$ and in fact, $\psi(\nabla)$ lies in the kernel of the connection on $\op{End}(\cl{E})\otimes F^*_{X/S}\Omega^1_{X^{(p)/S}}$, which is induced by the canonical connection on $F^*_{X/S}\Omega^1_{X^{(p)/S}}$ and the connection $\nabla^{\op{End}}$, which is induced by $\nabla$.

Cartier's theorem provides a characterization of connections that have zero $p$-curvature.

\begin{theorem}[Cartier, Theorem 5.1 in \cite{Katz1970}]
	With notations as before, there is an equivalence of categories between the category of quasi-coherent sheaves on $X^{(p)}$ and the full subcategory of $\op{MIC}(X/S)$ consisting of objects $(\cl E,\nabla)$ whose $p$-curvature is equal to zero. 
	
	More explicitly, the equivalence is given in the following way: given a quasi-coherent sheaf $\cl F$ on $X^{(p)}$, 
	there is a unique integrable $S$-connection 
	$\nabla_{\op{can}}$ on $F^*_{X/S}(\cl F)$, which has $p$-curvature zero and is such that 
	$\cl F\simeq (F^*_{X/S}(\cl F))^{\nabla_{\op{can}}}$.
	
	Conversely, given $(\cl E,\nabla)\in \op{MIC}(X/S)$ of zero $p$-curvature, $\cl E^{\nabla}$ is a quasi-coherent sheaf on $X^{(p)}$.
\end{theorem}

We now focus in particular on the case $S=\op{Spec}\clo{\mathbb F_q}$. We denote by $\psi(\nabla)$ the corresponding section in $\op{H}^0(X,\op{End}(\cl{E})\otimes F^*_{X/S}\Omega^1_{X^{(p)}/\clo{\mathbb F_q}})$ and note that by \cite[Proposition 5.2.3]{Katz1970}, the $p$-curvature of any connection $(E,\nabla)$ is flat under the tensor product of the canonical connection on $F^*_{X/S}\Omega^1_{X^{(p)}}$ and the one of $\op{End}(\nabla)$ on $\op{End}(E)$. 

The Hitchin map, \cite{Hitchin1987} and \cite[p.12]{Langer2014}, is then defined by mapping an integrable connection to the characteristic polynomial of its $p$-curvature, which by the above discussion will have coefficients in the global symmetric forms on $X^{(p)}$:
\[ \upchi(\psi(\nabla)) : = \op{det}(-\psi(\nabla) + \mu \op{Id}) = \mu^r - a_1\mu^{r-1}+\cdots +(-1)^ra_r \]
with $ a_i\in \op{H}^0(X^{(p)},\op{Sym}^i(\Omega^1_{X^{(p)}})))$ and $r=\op{rank}\cl{E}$.

In the case of a line bundle with connection on the curve $C$, we get a map, see for example \cite[Proposition 3.2]{Laszlo2001},\cite[Section 4]{Braverman2007}, \cite[Definitions 3.12 and 3.16]{Groechenig2016}:
\[ \upchi(\psi(\nabla)): \op{Pic}^{\nabla}(C) \to \cl{A}^1:=\op{H}^0(C^{(p)},\Omega^1_{C^{(p)}}),
\]
which is is known to be proper, \cite[Theorem 3.8]{Langer2014}.

\subsection{Isocrystals}
On the other hand, we say that a point $[(\cl{L},\nabla)]$ of $\op{Pic}^{\nabla}(C_W)$ represents an isocrystal on $C$ if the associated pair $(\cl{L}_{\clo{\mathbb F_q}},\nabla_{\clo{\mathbb F_q}})$ has nilpotent $p$-curvature.
This condition is equivalent to requiring that the characteristic polynomial of the $p$-curvature is zero or that $\upchi([(\cl{L}_{\clo{\mathbb F_q}},\nabla_{\clo{\mathbb F_q}} )])=0$, by definition of the Hitchin map.  Equivalently $(\cl{L}_{\clo{\mathbb F_q}},\nabla_{\clo{\mathbb F_q}} )$ is in the fiber above zero of the Hitchin map, which   we denote  by $\pic^{\nabla}(C)^{\psi=0}$. By the Cartier isomorphism, \cite[Theorem 5.1]{Katz1970}, we have that this fiber is isomorphic to $\pic^0(C^{(p)})$.

Since we are working on the perfect field $\clo{\mathbb F_q}$, we have that the Frobenius on it is an isomorphism. Then also $\pic^0(C^{(p)})$ and $\pic^0(C)$ are isomorphic as schemes. If we worked over $\mathbb F_p$, they would even be isomorphic as $\mathbb F_p$-schemes.

\subsection{A conjecture of Deligne}\label{SectionDeligne}
In \cite[Section 2.17]{Deligne2015} Deligne considers the following situation: let $M_K$, and respectively $M_{\clo K}$ denote the moduli space of vector bundles of rank $r$ on a smooth curve $C_K$, respectively $C_{\clo K}$, endowed with an integrable connection. The space $M_{\clo K}$ is obtained by $M_{K}$ by extension of scalars. Denote as well by $E_r$ the set of isomorphism classes of $\clo{\mathbb Q}_\ell$ - irreducible lisse sheaves of rank $r$ on $C$. For a fixed embedding $\iota: \clo K \to \mathbb C$, denote by $\Sigma$ the Riemann surface obtained by $C_{\clo K}$, by extending the scalars to $\mathbb C$. The space $M_{\mathbb C}$ corresponding to $E_r$ is induced by $M_{\clo K}$ by extension of scalars and the $\ell'$-adic cohomology, for $\ell'\neq p$ a prime number, of $M_{\clo K}$ is isomorphic to the cohomology of $M_{\mathbb C}$ with $\mathbb{Q}_{\ell'}$ coefficients. (Note that all cohomology groups below will denote $\ell'$-adic cohomology groups, denoted by $\op{H}^*(-)$.)

\begin{conj*}[{\cite[Conjecture 2.18]{Deligne2015}}]\label{ConjDeligne}
	The cohomology of $M_{\clo K}$ admits an endomorphism $V^*$, such that for all $n\ge 1$, the number $N_n$ of fixed points of $V^*$ on $E_r$ is given by
	\[ N_n=\sum_i(-1)^i\op{Tr}(V^{*n}, \op{H}^i(M_{\clo K})).   \]
\end{conj*}
In general, we do not have a Frobenius action on the moduli space of vector bundles with integrable connection. Indeed, Deligne expects that there should exist an open subspace $M^0_{\clo K}$ inside the Berkovich analytification of $M_{\clo K}$, which corresponds to the sublocus of isocrystals and such that it has the following two properties:
\begin{enumerate}[label={(\arabic*)}]
	\item the restriction morphism $\op{H}^*(M_{\clo K}) = \op{H}^*(M^{\op{an}}_{\clo K})\to \op{H}^*(M^0_{\clo K})$ is an isomorphism and
	\item a crystalline interpretation of $M^0$ allows us to define $V= \op{Frob}^*: M^0_{\clo K}\to M^0_{\clo K}$, which induces $V^*$ on cohomology.
\end{enumerate}
Then the number of fixed points of the action of $V^n$ would be given by
\begin{equation}\label{DeligneFixedPoints}
N_n:=\sum (-1)^n\op{Tr}(V^{*,n},\op{H}^i(M_{\clo K})).
\end{equation}
The conjectured morphism $V:M^0_{\clo K}\to M^0_{\clo K}$ should send $M^0_{\clo K}$ into a proper open subset of $M^0_{\clo K}$, and one can take ordinary cohomology, instead of cohomology with compact support in the Lefschetz Trace formula.

Deligne provided in \cite[Example 2.19]{Deligne2015} an example for this conjecture in the rank 1 case, without explaining why the properties (1) and (2) from above are fulfilled. 
\begin{example}[{\cite[Example 2.19 and Proposition 2.20]{Deligne2015}}]
	Denote the absolute Frobenius of $C_0$ by $F_{C_0}$ and that of $\pic^{0}(C_0)$ by $F_{\pic^{0}(C_0)}$. Pullback by the absolute Frobenius of $C_0$ defines $F^*_{C_0}$, an endomorphism of $\pic^0(C_0)$ for which \cite[(2.2.2)]{Deligne2015}:
	\begin{equation}
	F^*_{C_0}\circ F_{\pic^0(C_0)} =F_{\pic^0(C_0)}\circ F^*_{C_0} = q.
	\end{equation}
	Because of this we denote $F^*_{C_0}$ by $V$ (Verschiebung).

	Extending scalars to the algebraic closure $\clo{\mathbb F_q}$, we obtain the corresponding endomorphisms $F_C$  on $C$: $F_C$ maps a point of $C$ with affine coordinates $(x_1,\dots ,x_n)$ to the point with coordinates $(x_1^q,\dots,x_n^q)$.   We call this the \textit{Frobenius endomorphism} of $C$. The fixed points of the action of this map are exactly the $\mathbb F_q$ points of $C$, and are computed by the Lefschetz Trace formula, given below in (\ref{traceForm}). 
	
	The morphism $F_{C}$ induces an endomorphism on cohomology:
	\[ \op{H}^i(C,\mathbb{Q}_\ell)\rightarrow \op{H}^i(C,\mathbb{Q}_\ell)\]

	In the rank 1 case, the moduli space $M_{\clo K}$ is the same as $\pic^\nabla(C_{\clo K})$, which, as already discussed in Section \ref{SectionUnivExtension}, is the universal extension of $\pic^0(C_{\clo K})$.
	By homotopy invariance, it is true that
	\begin{equation}\label{ExampleDelEq1}
	\op{H}^*(\pic^0(C_{\clo K}))\overset{\simeq}{\longrightarrow}\op{H}^*(\pic^\nabla(C_{\clo K}))
	\end{equation} 
	while we also have a natural isomorphism by smooth base change
	\begin{equation}\label{ExampleDelEq2}
	\op{H}^*(\pic^0(C))\overset{\simeq}{\longrightarrow} \op{H}^*(\pic^0(C_{\clo K})).
	\end{equation}
	On $\clo{\mathbb{F}}_q$, we have an endomorphism $V:\pic^0(C)\to \pic^0(C)$ induced by functoriality from the Frobenius endomorphism of $C$ and a pullback on cohomology:
	\begin{equation}\label{ExampleDelEq3}
	V^*: \op{H}^*(\pic^0(C))\longrightarrow\op{H}^*(\pic^0(C))
	\end{equation}
	and because of (\ref{ExampleDelEq1}) and (\ref{ExampleDelEq2}) this defines
	\begin{equation}\label{ExampleDelEq4}
	V^*: \op{H}^*\pic^{\nabla}(C_{\clo K})\longrightarrow\op{H}^*\pic^{\nabla}(C_{\clo K}).
	\end{equation}
	In \cite[Proposition 2.20]{Deligne2015}, Deligne further computes the fixed points of this action for $n=1$ in (\ref{DeligneFixedPoints}).  Note that one can reduce to this case, taking an extension of scalars from $\mathbb F_q$ to $\mathbb F_{q^n}$. As also explained in \cite[(6.1)]{Deligne2015}, the number of fixed points by the Frobenius correspond to the points fixed by $V$ on $E_1$, and as we see later by our crystalline interpretation, this number should be the number of $F$-isocrystals defined over $\mathbb{F}_q$.

	The Frobenius endomorphism $F_{\pic^0(C)}$ and  $V$ on $\pic^0(C)$ are transpose to each other, since $\pic^0(C)$ is auto-dual, and we therefore obtain, \cite[Proposition 2.20]{Deligne2015}:
	\begin{equation}
	\op{Tr}(V^{*},\op{H}^i(\pic^{0}(C))) = \op{Tr}(F_{\pic^0(C)}^{*},\op{H}^i(\pic^0(C)))
	\end{equation}
	and
	\begin{equation}\label{traceForm}
	\sum(-1)^i  \op{Tr}(V^{*},\op{H}^i(\pic^{0}(C))) = |\pic^0(C_0)(\mathbb F_q)|.
	\end{equation}
\end{example}
Our goal  is mainly to explain Deligne's example and provide a crystalline interpretation of it. We revisit this at the end of the article, in Example \ref{ExampleDeligne}.
In the spirit of the  conjecture, we are considering the Berkovich analytification of $\pic^{\nabla}(C_W)_{K}$ and its base change to $\clo{K}$, $\pic^{\nabla}(C_W)_{\clo K}$. 
As recalled in Proposition \ref{DefReduc}, we have the anti-continuous reduction map
\begin{equation}\label{reductionMap}
\op{\pi}: \widehat{(\pic^{\nabla}(C_W))}_{K}\to \pic^{\nabla}(C)
\end{equation}
where $\widehat{(\pic^{\nabla}(C_W))}_{K}\hookrightarrow \pic^{\nabla}(C)^{\op{an}}_{ K}$, which would be an isomorphism if $\pic^{\nabla}(C)_{K}$ was proper, see Remark \ref{RemBerkGen}. We denote the inverse image of $\pic^{\nabla}(C)^{\psi=0}$, which is also called the \textit{tube} of $\pic^{\nabla}(C)^{\psi=0}$ inside the analytification, by $]\pic^{\nabla}(C)^{\psi=0}[$. It is an open subset of $\pic^{\nabla}(C)^{\op{an}}_{K}$ and is actually isomorphic to
\[\widehat{(\pic^{\nabla}(C_W)}_{\pic^{\nabla}(C)^{\psi=0}})_{K} ,\]
the generic fiber of the completion of $\pic^{\nabla}(C_W)$ along the closed subset $\pic^{\nabla}(C)^{\psi=0}$, see \ref{PropInverseImageBerk}.

\section{Comparison of cohomology groups.}\label{SectionModCoh}
In order to verify the aforementioned conjecture of Deligne, we have to compare the cohomology groups of $M^0_{\clo K}$ and $M_{\clo K}$ of Conjecture \ref{ConjDeligne}, which in our case are $]\pic^{\nabla}(C)^{\psi=0}[_{\clo K}$ and $\pic^{\nabla}(C)^{\op{an}}_{\clo K}$. 
As before, cohomology here means $\ell$-adic cohomology.

We can use the results recalled in Section \ref{SubsectionBerCoh} in order to prove the following comparison:
\begin{proposition}\label{PropCoh}
	There is an isomorphism of cohomology groups
	\[ \op{H}^*(\pic^{\nabla}(C_{\clo K})^{\op{an}},\mathbb Q_{\ell})\simeq
	\op{H}^* (]\pic^{\nabla}(C)^{\psi=0}[_{\clo K},\mathbb Q_{\ell})\]
\end{proposition}
\begin{proof}
	By Corollary \ref{CorBerkCoh} we have
	\begin{equation}\op{H}^* (]\pic^{\nabla}(C)^{\psi=0}[,\mathbb Q_{\ell}) \simeq
	\op{H}^* (\pic^{\nabla}(C)^{\psi=0},\mathbb Q_{\ell})
	\end{equation}
	but by the smooth base change theorem, see \cite[Theorem 20.1]{Milne08lectureson} 
	\begin{equation}\op{H}^* (]\pic^{\nabla}(C)^{\psi=0}[,\mathbb Q_{\ell}) \simeq
	\op{H}^* (]\pic^{\nabla}(C)^{\psi=0}[_{\clo K},\mathbb Q_{\ell})
	\end{equation}
	and as explained in Section \ref{SectionModIsoc},
	\begin{equation}\label{morFib} \op{H}^* (\pic^{\nabla}(C)^{\psi=0},\mathbb Q_{\ell}) \simeq \op{H}^* (\pic^{0}(C^{(p)}),\mathbb Q_{\ell}) \simeq \op{H}^* (\pic^{0}(C),\mathbb Q_{\ell}). \end{equation}
	By Theorem 20.5 of Milne \cite{Milne08lectureson} we have that, since $\pic^0(C_W)$ is proper over $\op{Spec}W$, it is also true that
	\begin{equation} \op{H}^* (\pic^{0}(C),\mathbb Q_{\ell})\simeq \op{H}^* (\pic^{0}(C)_{\clo K},\mathbb Q_{\ell}). \end{equation}
	However, since $\pic^{\nabla}(C)_{\clo K}$ is the universal extension of $\pic^{0}(C)_{\clo K}$, by Section \ref{SectionUnivExtension}, we have
	\begin{equation}  \op{H}^* (\pic^{0}(C)_{\clo K},\mathbb Q_{\ell})
	\simeq \op{H}^* (\pic^{\nabla}(C_{\clo K}),\mathbb Q_{\ell}) \end{equation}
	Finally, the latter cohomology group is isomorphic to $ \op{H}^* (\pic^{\nabla}(C_{\clo K})^{\op{an}},\mathbb Q_{\ell})$, by \cite[Corollary 7.5.4]{Berkovich1993}.
\end{proof}
\section{Frobenius action}\label{SectionModFrob}
The goal of this section is to show that there is a Frobenius action on $]\pic^{\nabla}(C)^{\psi=0}[$, which induces a Frobenius pullback morphism on cohomology. 
In general, if we have an $S$-scheme $M$ which represents a functor $G$, then to define a morphism $M\to M$, it is enough to define such a morphism $G(S')\to G(S')$ for all $S$-schemes $S'$.

With this in mind, we obtain in this section a moduli interpretation of $]\pic^{\nabla}(C)^{\psi=0}[$ and this yields, because of its nature, a natural Frobenius action. 

\subsection{Functorial interpretation}
We are interested in the functorial interpretation of the subset of isocrystals inside the analytic space $ \pic^{\nabla}(C)^{\op{an}}_{K}$. We consider here the rigid analytification.
Indeed, the Berkovich and rigid analytifications of an algebraic variety are compatible constructions, as explained in Remark \ref{RemBerkRig}. 

From now on, $(-)^{\op{an}}$ denotes the rigid analytification functor and $(-)_{K}$ denotes the rigid analytic generic fiber.

By Remark \ref{AnalytPic}, we can define the functor $\textbf{Pic}^{\nabla}_{C_{K}^{\op{an}}/{K}, x^{\op{an}}}$ which is represented by the analytification of $\pic^{\nabla}(C_{K})$.
Our goal is therefore to characterize $]\pic^{\nabla}(C)^{\psi=0}[$ as a subfunctor of  $\textbf{Pic}^{\nabla}_{C_{K}^{\op{an}}/{K}}$. For this it is enough to describe the set \[\op{Hom}(S,]\pic^{\nabla}(C)^{\psi=0}[ ) \] for a rigid analytic space $S$. This will be the set of line bundles with connection on $S\times_K C^{\op{an}}_K\to S$, as before, with some extra property. Hence, it is enough to assume $S$ is an affinoid and equal to $\op{Sp}(A)\otimes {K}=\op{Sp}A_{ K}$, for some $W$-algebra $A$.

A morphism of rigid analytic spaces 
\begin{equation}\label{morRig}
S=\op{Sp}A_{K} \xrightarrow{\phi_{K}} ]\pic^{\nabla}(C)^{\psi=0}[=
\widehat{(\pic^{\nabla}(C_W)}_{\pic^{\nabla}(C)^{\psi=0}})_{ K}
\end{equation} extends as recalled in Remark \ref{FormModProof} to a morphism of formal schemes 
\begin{equation}\label{morFor}
\cl S' \xrightarrow{\phi'} \widehat{\pic^{\nabla}(C_W)}_{\pic^{\nabla}(C)^{\psi=0}}
\end{equation}
where $\cl S'\rightarrow \cl S=\op{Spf}(A)$ is an admissible formal blow-up.
For $C_W\to \op{Spf}W$ the formal lift of $C$, having a morphism as in (\ref{morFor}) means that the morphism
$\cl S' \to \pic^{\nabla}(C_W)$ factors as

\begin{equation}\label{morAdmBlow}
\begin{tikzcd}[row sep=2em]
\mathscr{S}' \arrow{rr} \arrow{dr} && \op{Pic}^{\nabla}(C_W)\\
&\widehat{\pic^{\nabla}(C_W)}_{\pic^{\nabla}(C)^{\psi=0}}\arrow{ur}
\end{tikzcd}
\end{equation} 
which in turns yields
\begin{equation}
\begin{tikzcd}[row sep=2em]
S'_{\clo{\mathbb{F}_q}} \arrow{rr} \arrow{dr} && \op{Pic}^{\nabla}(C)\\
&\op{Pic}^{\nabla}(C)^{\psi=0} \arrow{ur}
\end{tikzcd}
\end{equation}
This is equivalent to saying that the pair $(\cl L,\nabla)$ defined by (\ref{morFor}) is actually an isocrystal on $\cl S' \times C_W\to \cl S'$.
\begin{remark}
	This construction is independent of the choice of the admissible blow-up: suppose there is another admissible blow-up $\cl S''$ of $\cl S$ such that the morphism (\ref{morRig}) lifts to a morphism of formal schemes 
	\begin{equation}
	\cl S''\xrightarrow{\phi''} \widehat{\pic^{\nabla}(C_W)}_{\pic^{\nabla}(C)^{\psi=0}}.
	\end{equation}
	Then as in Remark \ref{FormModProof} we have that $\phi'$ and $\phi''$ coincide, since the rigid analytic generic fibers $\cl S_{K}'$ and $\cl S_{K}''$ coincide; they are isomorphic to $S$. 
\end{remark}
On the other hand, if we start with a pair of a line bundle with connection on $S\times C_{K}^{\op{an}}\to S$ for an affinoid $S$, which is an isocrystal on $\cl S' \times C_W\to \cl S'$, for $\cl S'\to \cl S$ an admissible blow-up of $\cl S$, this means by definition that we have a morphism as in  (\ref{morAdmBlow}). Taking the associated map between the rigid generic fibers, we obtain a morphism \[S' \to  ]\pic^{\nabla}(C)^{\psi=0}[. \] By \cite[(a) in p.307]{Bosch1993}, recalled in Remark \ref{FormModProof}, there is an isomorphism between the rigid generic fiber of $\cl S$ and $\cl S'$, which means that this morphism is an element of $\op{Hom}(S,]\pic^{\nabla}(C)^{\psi=0}[ )$. Note that the category of isocrystals on a formal lift  of a scheme depends only on its reduction, because of the nilpotence condition.

By definition of taking the extension of a morphism between rigid spaces to a morphism of the associated formal models and that of taking the rigid generic fiber of a morphism of formal spaces, these two constructions are inverse to each other.
\begin{proposition}
	We obtain therefore an isomorphism between the set 
	\begin{equation}
	\op{Hom}(S,]\pic^{\nabla}(C)^{\psi=0}[ )
	\end{equation} and the set 
	\begin{multline}\label{setHom}
	\{ (\cl{L},\nabla)\text{ line bundles with integrable connection on }S\times C^{\op{an}}_K \text{ which are}\\\text{ isocrystals on }
	\cl S'\times C_W\to \cl S' \text{, for } S'\text{ an admissible formal blow-up of }S\}
	\end{multline}
	We denote the latter by $]\pic^{\nabla}(C)^{\psi=0}[(S)$.
\end{proposition}

\begin{remark}
	Note also that this isomorphism is functorial: if $S\to T$ is a morphism of affinoid spaces, we have a map $]\pic^{\nabla}(C)^{\psi=0}[(S)\to ]\pic^{\nabla}(C)^{\psi=0}[(T)$. (\cite{BerPre} : the category of overconvergent isocrystals on $(X,S)$ is functorial on $(X,S)$).
\end{remark}
\subsection{Frobenius}

In order to define a Frobenius pull-back morphism on $]\pic^{\nabla}(C)^{\psi=0}[$, it is enough to define it on $]\pic^{\nabla}(C)^{\psi=0}[(S)$. Because of the previous discussion, given a $(\cl L,\nabla)=:L$ in (\ref{setHom}), and choosing a lift $\phi$ of the relative Frobenius of $S'\times C\to S'$ to $\cl S'\times C_W\to C_W$, we have a well-defined isocrystal  $\phi^*L$ on $\cl S'\times C_W$.

\begin{example}[{\cite[Example 2.19]{Deligne2015}}]\label{ExampleDeligne}
	In light of our interpretation of the subset $]\pic^{\nabla}(C)^{\psi=0}[$, we can understand why Deligne used the cohomology of $\pic^0$ for the comparison on cohomology and why indeed this example gives an affirmative answer to his conjecture in rank 1. By his computation, we thus obtain the number of fixed points of the Frobenius action, which correspond to the isocrystals with Frobenius structure.
\end{example}

\normalem
\printbibliography
\end{document}